\documentclass{amsart}

\usepackage{amsmath, amsthm, amssymb, mathrsfs, sansmath}
\usepackage[colorlinks=true, citecolor=blue]{hyperref}
\usepackage{epsfig, graphicx}
\usepackage{verbatim,setspace}

\newcommand{\D}		{\mathbb{D}}

\newcommand{\C}		{\mathbb{C}}
\newcommand{\N}		{\mathbb{N}}

\newcommand{\cws}{\stackrel{*}{\to}}
\DeclareMathOperator{\const}{\mathrm{const.}}
\DeclareMathOperator{\dist}{\mathrm{dist}}
\DeclareMathOperator{\diam}{\mathrm{diam}}
\DeclareMathOperator{\supp}{\mathrm{supp}}
\DeclareMathOperator{\im}{\mathrm{Im}}
\DeclareMathOperator{\re}{\mathrm{Re}}
\DeclareMathOperator{\cp}{\mathrm{cp}}

\DeclareMathOperator{\prob}{\mathsf{Prob}}
\DeclareMathOperator{\e}{\mathsf{E}}

\newtheorem{theorem}{Theorem}
\newtheorem{proposition}[theorem]{Proposition}
\newtheorem{corollary}[theorem]{Corollary}
\newtheorem{lemma}[theorem]{Lemma}

\begin{document}

\title[Large Deviations and Linear Statistics for Potential Theoretic Ensembles]{Large Deviations and Linear Statistics for Potential Theoretic Ensembles Associated with Regular Closed Sets}

\author[M. Yattselev]{Maxim L. Yattselev}

\address{Department of Mathematics, University of Oregon, Eugene, OR, 97403}

\email{maximy@uoregon.edu}



\maketitle

\begin{abstract}
A two-dimensional statistical model of $N$ charged particles interacting via logarithmic repulsion in the presence of an oppositely charged regular closed region $K$ whose charge density is determined by its equilibrium potential at an inverse temperature $\beta$ is investigated. When the charge on the region, $s$, is greater than $N$, the particles accumulate in a neighborhood of the boundary of $K$, and form a  point process in the complex plane. We describe the weak$^*$ limits of the joint intensities of this point process and show that it is exponentially likely to find the process in a neighborhood of the equilibrium measure for $K$.
\end{abstract}

\maketitle

\section{Introduction}

In two-dimensional electrostatics, charged particles are identified with points in the extended complex plane.  The potential energy of a system of two like charged particles located at $z, w \in \C$ is proportional to $-\log| z - w |$.  More generally, if $z_1, z_2, \ldots, z_N$ are the locations of $N$ identically charged particles, then $\{z_1,\ldots,z_N\}$ determines the state of the system and the potential energy of this state is given by
\[
 -\sum_{m < n} \log| z_n - z_m |.
\]
The energy is minimized when the particles are all at $\infty$.  In order for the system to be found in a state where the particles are at finite positions, there needs to be a potential (or other obstructions) which repels the particles from $\infty$.  We represent this field by $V$ so that the interaction energy between a particle located at $z$ and the field is given by $V(z)$.  The total potential energy of the system comprised of the $N$ particles in the field is given by
\[
E(z_1,\ldots,z_N) = \sum_{n=1}^N V(z_n) - \sum_{m < n} \log| z_n - z_m |.
\]

The system is assumed to be in contact with a heat reservoir so that the energy of the system is variable, but the temperature is fixed.  In this setting, $\beta$ denotes the reciprocal of the temperature, and the Boltzmann factor for the state $\{z_1,\ldots,z_N\}$ is given by
\[
e^{-\beta E(z_1,\ldots,z_N)} = \bigg\{\prod_{n=1}^N e^{-\beta V(z_n)} \bigg\} \prod_{m < n} |z_n - z_m|^{\beta}.
\]
This quantity gives the relative density of states, so that the probability (density) of finding the system in state $\{z_1,\ldots,z_N\}$ is
given by
\[
Z_{N,\beta,V}^{-1} e^{-\beta E(z_1,\ldots,z_N)}, \quad Z_{N,\beta,V} = \int_{\C^N} e^{-\beta E(z_1,\ldots,z_N)} \mathrm{d}A^{\otimes N}(z_1,\ldots,z_N).
\]
Let $\omega:=\sum_{n=1}^N\delta_{z_n}$ be the empirical measure associated with a state $\{z_1,\ldots,z_N\}$. In this work we take $V$ to be minus the equilibrium potential for some regular closed region $K$ and study the following two questions. First: what is probability to find the system (empirical measure $\omega$) close to a given Borel measure on the complex plane (large deviation principle); second: what is the limiting behavior of the marginal probabilities of the probability density function for this system (linear statistics).

These questions have been obviously considered before. In \cite{Voic93,Voic94,BenAGui97}, see also \cite{HiaiPetz,AndersonGuionnetZeitouni}, the large deviation principle was shown for the case of the external field $V$ being identically $+\infty$ of the real line (thus, all the charges are confined to the real line) and satisfying $\lim_{|x|\to\infty}\left(\log|x|-\epsilon V(x)\right)=-\infty$ for any $\epsilon>0$. This work was further extended to the case $V(z)=|z|^2$ without confinement to the real line in \cite{BenAZei98,HiPetz_ADEMP98}. In \cite{HiPetz00}, the large deviation principle was shown to holds for particles restricted to the unit circle, i.e., $V$ is continuous on the unit circle and $V\equiv+\infty$ of it. 

As to the linear statistics, the case of particles confined to the real line and $V$ being polynomial of even degree with positive leading coefficient was studied in \cite{Joh98}. The case where particles are restricted to a compact subset of the complex plane interacting in the presence of a continuous field was treated in \cite{ElFel05}. No confinement case with regular fields satisfying $V(z)\geq(1+\epsilon)\log(1+|z|)$ for some $\epsilon>0$ was considered in \cite{uHedMak}.

\section{Main Results}

\subsection{Potential Theoretic Setting}

For any probability Borel measure on $\C$, say $\nu$, set
\[
I[\nu] := \int\log\frac{1}{|z-u|}\mathrm{d}\nu^{\otimes2}(z,u)
\]
to be its \emph{logarithmic energy} (negative free entropy), where $\mathrm{d}\nu^{\otimes n}(z_1,,\ldots,z_n):=\mathrm{d}\nu(z_1)\cdots\mathrm{d}\nu(z_n)$. For any compact set $K$ the \emph{logarithmic capacity} of $K$ is defined by
\[
\cp(K) := \exp\left\{-\inf_{\supp(\nu)\subseteq K} I[\nu]\right\}.
\]
It is known that either $\cp(K)=0$ ($K$ is \emph{polar}) or there exists the unique measure $\omega_K$, the \emph{logarithmic equilibrium distribution} on $K$, that realizes the infimum. That is, 
\begin{equation}
\label{capacity}
\cp(K) = \exp\bigg\{-I[\omega_K]\bigg\}.
\end{equation}
The \emph{logarithmic potential} of $\omega_K$, that is,
\[
V^{\omega_K}(z) := \int \log\frac1{|z-u|}\mathrm{d}\omega_K(u),
\]
is equal to $I[\omega_K]$ \emph{quasi everywhere} (up to a polar set) on $K$ and is at most as large everywhere in the complex plane. The set $K$ is called \emph{regular} with respect to the Dirichlet problem if $V^{\omega_K}=I[\omega_K]$ everywhere on $K$. The \emph{Green function} with a pole at infinity for the unbounded component of $K^c$, the complement of $K$, is defined by
\[
g_K := I[\omega] - V^{\omega_K}.
\]
It is a non-negative harmonic function in $K^c\setminus\{\infty\}$ with a logarithmic singularity at infinity. Moreover, it is zero q.e. on $K$ and is, in fact, continuous if $K$ is regular.

Let $s$ and $N$ be two parameters such that $s>N$. We always assume that $s$ and $N$ scale in such a fashion that the limit of $s^{-1}N$ as $N$ tends to infinity exists and define the following energy functional
\begin{equation}
\label{energy}
I_\ell[\nu] := I[\nu] + \frac2\ell\int g_K\mathrm{d}\nu, \qquad \ell:=\lim_{N\to\infty}s^{-1}N,
\end{equation}
where it is understood that $I_0[\nu]=\infty$ for all $\nu$ such that $\supp(\nu)\cap K^c\neq\varnothing$.
\begin{proposition}
\label{prop:minenergy}
Let $K$ be a compact set with connected complement which is regular with respect to the Dirichlet problem. For all $\ell\in[0,1]$, it holds that
\begin{equation}
\label{inequalities}
I[\omega_K] < I_\ell[\nu]
\end{equation}
for any compactly supported probability Borel measure $\nu\neq\omega_K$.
\end{proposition}
Clearly, $I_\ell[\nu]=I[\nu]$ for any measure $\nu$ supported on $K$ as $g_K\equiv0$ on $K$. Hence, $I_\ell[\omega_K]=I[\omega_K]$ and therefore $\omega_K$ is the unique minimizer of the weighted energy functional $I_\ell$ for any $\ell\in[0,1]$.

\subsection{A Model for Random Configurations}

Let $K$ be a compact set with connected complement which is regular with respect to the Dirichlet problem. In this paper we investigate random configurations whose joint density is given by  
\begin{equation}
\label{Omega}
\Omega_{N,s,\beta}(z_1,\ldots,z_N) :=\frac{1}{Z_{N,s,\beta}}\exp\left\{-\beta s\sum_{n=1}^Ng_K(z_n) \right\} \prod_{m < n} |z_n - z_m |^\beta,
\end{equation}
where $(N,s,\beta)$ is a triple of numbers such that
\begin{equation}
\label{snbeta}
\beta(s-N+1)>2+c_0
\end{equation}
for some fixed $c_0>0$, and $Z_{N,s,\beta}$ is a normalizing factor that turns $\Omega_{N,s,\beta}$ into a probability density function. Clearly,
\begin{equation}
\label{ZNsbeta}
Z_{N,s,\beta} = \int_{\C^N} \exp\left\{ -\beta s\sum_{n=1}^N g_K(z_n) \right\} \,\prod_{m < n} |z_n - z_m|^\beta \, \mathrm{d}A^{\otimes N}(z_1,\ldots,z_N),
\end{equation}
where $\mathrm{d}A$ stands for the Lebesgue measure on $\C$. 

Besides connections to electrostatics and random matrix theory, the results below are also motivated by number theory. Let $K$ be such that $\cp(K)=1$ and $p$ be a polynomial. The Mahler of $p$ with respect to $K$ is defined by
\[
M_K(p) := \exp\left\{\int\log|p|\mathrm{d}\omega_K\right\}=\exp\bigg\{a_p\sum_{z:~p(z)=0}g_K(z)\bigg\},
\]
where $a_p$ is the leading coefficient of $p$. When $K=\overline\D$, it is known that $\mathrm{d}\omega_K(z)=\frac1{2\pi}|\mathrm{d}z|$, and therefore $M:=M_{\overline\D}$ is simply the classical Mahler measure. In \cite{ChVaa01}, the bound for the number of polynomials with integer coefficients of degree at most $N$ such that $M(f)\leq\const$ was derived. The main term of the asymptotics for this bound came from $Z_{N,N+1,2}$ defined by \eqref{ZNsbeta} with $K=\overline\D$. Moreover, it was shown that $Z_{N,s,2}$ is a rational function of $s$ with poles at every positive integer less or equal to $N$. The cases where $K$ is an ellipse and $E=[-2,2]$ were also investigate in \cite{Sin04,Sin05}.

\subsection{Large Deviation Principle}

Let $\boldsymbol\eta=\{\eta_1,\ldots,\eta_N\}$ be a random configuration chosen according to the law $\Omega_{N,s,\beta}$. That is, the probability that $\eta_j\in O$, $j\in\{1,\ldots,N\}$, is equal to $\int_{O^N}\Omega_{N,s,\beta}\mathrm{d}A^{\otimes N}$ for any open set $O\subseteq\overline\C$. To any such configuration we associate the empirical measure $\omega_{\boldsymbol\eta}$ defined as
\[
\omega_{\boldsymbol\eta} := \frac1N\sum_{k=1}^N\delta_{\eta_k},
\]
where $\delta_z$ is the classical Dirac delta distribution with the unit mass at $z$.

Let $\nu$ and $\mu$ be two probability Borel measures on $\C$. Then the distance between them is defined by
\[
\dist(\nu,\mu) = \sup_f\left|\int f\mathrm{d}\nu-\int f\mathrm{d}\mu\right|,
\]
where the supremum is taken over all functions $f$ that are bounded by 1 in modulus and satisfy the Lipschitz condition with constant 1 on $\supp(\nu)\cup\sup(\mu)$. For measures supported on a compact set it holds that $\dist(\nu,\nu_n)\to0$ as $n\to\infty$ if and only if $\nu_n\cws\nu$ as $n\to\infty$, where $\cws$ stands for the convergence in the weak$^*$ topology of measures ($\int f\mathrm{d}\nu_n\to\int f\mathrm{d}\nu$ for any continuous function $f$).

\begin{theorem}
\label{thm:partition}
Let $K$ be a compact set with connected complement which is regular with respect to the Dirichlet problem and such that $K=\overline{K^\circ}$. Given \eqref{snbeta}, it holds that
\begin{equation}
\label{partitionlimit}
\lim_{N\to\infty}\frac{1}{N^2}\log Z_{N,s,\beta} = -\frac\beta2 I[\omega_K].
\end{equation}
Moreover, for any probability Borel measure $\nu$, $\supp(\nu)\subset\C$, it is true that
\begin{equation}
\label{largedeviation}
\lim_{\epsilon\to0}\lim_{N\to\infty} \frac{1}{N^2}\log \prob\left\{\dist\left(\nu,\omega_{\boldsymbol\eta}\right)<\epsilon\right\} = -\frac\beta2\bigg(I_\ell[\nu]-I[\omega_K]\bigg).
\end{equation}
\end{theorem}

Equations \eqref{inequalities} and \eqref{largedeviation} yield that the probability to find $\omega_{\boldsymbol\eta}$ in a small enough neighborhood of $\nu$ is subexponentially small if $I_\ell[\nu]=\infty$ and is exponentially small  if $I_\ell[\nu]<\infty$ and $\dist\left(\nu,\omega_K\right)>0$.

The asymptotics in \eqref{partitionlimit} can be improved if we restrict the attention to Jordan domains with smooth boundary and $\beta=2$.
\begin{proposition}
\label{prop:beta2}
Let $K$ be a Jordan domain whose boundary $\partial K$ is a Jordan curve of class\footnote{The arclength function of $\partial K$ is continuously differentiable as a periodic function on the real line and its derivative is $\alpha$-H\"older continuous.} $C^{1,\alpha}$, $\alpha>1/2$. Then
\begin{equation}
\label{beta2}
\log Z_{N,s,2} = -N(N+1)I[\omega_K] +  \theta(\ell)N + \mathcal{O}(\log N),
\end{equation}
where $\theta(x)=\log\pi +1+x^{-1}(1-x)\log(1-x)$, which is a continuous increasing function on $[0,1]$ with values $\theta(0)=\log\pi$ and $\theta(1)=\log\pi+1$. When $s=\infty$, the term $\mathcal{O}(\log N)$ can be replaced by $\mathcal{O}(1)$.
\end{proposition}

\subsection{Linear Statistics}

The $n$-th marginal probability of $\Omega_{N,s,\beta}$, $n\in\{1,\ldots,N-1\}$, is defined by
\begin{equation}
\Omega^{(n)}_{N,s,\beta}(z_1,\ldots,z_n) := \int_{\C^{N-n}} \Omega_{N,s,\beta}(z_1,\ldots,z_N) \, \mathrm{d}A^{\otimes (N-n)}(z_{n+1},\ldots,z_N).
\end{equation}
Let $\boldsymbol\eta$ be a random configuration chosen according to the law $\Omega_{N,s,\beta}$ and $\omega_{\boldsymbol\eta}$ be the corresponding empirical measure. Then $\omega_{\boldsymbol\eta}$ can be considered as a point process on $\C$. It is known that the joint intensities of this point process are equal to $\big(N!/(N-n)!\big)\Omega^{(n)}_{N,s,\beta}$. That is, if $O_1,\ldots,O_n$ are mutually disjoint subsets of $\C$, then
\[
\e\left[\prod_{k=1}^n\omega_{\boldsymbol\eta}(O_k)\right] = \frac{N!}{(N-n)!}\int_{O_1\times\cdots\times O_n}\Omega^{(n)}_{N,s,\beta}\mathrm{d}A^{\otimes n},
\]
where $\e[\cdot]$ denotes the expected value of a random variable. The following theorem describes the weak$^*$ behavior of the measures $\Omega^{(n)}_{N,s,\beta}\mathrm{d}A^{\otimes n}$ as $N\to\infty$.

\begin{theorem}
\label{thm:cws}
Under the conditions of Theorem~\ref{thm:partition}, it holds that
\[
\lim_{N\to\infty}\int_{\C^n} f\Omega^{(n)}_{N,s,\beta}\mathrm{d}A^{\otimes n} = \int f\mathrm{d}\omega_K^{\otimes n}
\]
for each $f\in\mathrm{C}_b(\C^n)$, $n\in\N$, where $\mathrm{C}_b(\C^n)$ is the Banach space of bounded continuous functions on $\C^n$.
\end{theorem}

Theorem~\ref{thm:cws} in particular implies that
\[
\lim_{N\to\infty}\e\left[\int f\mathrm{d}\omega_{\boldsymbol\eta_N}\right] = \int f\mathrm{d}\omega_K, \quad f\in\mathrm{C}_b(\C),
\]
where $\boldsymbol\eta_N=\left\{\eta_1^N,\ldots,\eta_N^N\right\}$ is a random configuration chosen according to the law $\Omega_{N,s,\beta}$ since
\[
\e\left[f\left(\eta_j^N\right)\right]=\int f\Omega^{(1)}_{N,s,\beta}\mathrm{d}A
\]
for any $j\in\{1,\ldots,N\}$. More generally, let $\{j_1,\ldots,j_n\}\subseteq\{1,\ldots,N\}$ be a set of distinct indices and $f_1,\ldots,f_n\in\mathrm{C}_b(\C)$. Set $f(z_1,\ldots,z_n):=\prod_{k=1}^nf_k(z_k)$. Then $f\in\mathrm{C}_b(\C^n)$ and it holds that
\[
\e\left[f\left(\eta_{j_1}^N,\ldots,\eta_{j_n}^N\right)\right]=\int f\Omega^{(n)}_{N,s,\beta}\mathrm{d}A^{\otimes n} + \sum_{k=1}^{n-1} \left(\sum \int f\Omega^{(k)}_{N,s,\beta}\mathrm{d}A^{\otimes k}\right),
\]
where the inner sum is taken over all possible combinations of $n-k+1$ coordinates being equal. Observe that the first integral on the right-hand side of the equality above converges to $\prod_{k=1}^n\int f_k\mathrm{d}\omega_K$ as $N\to\infty$ by Theorem~\ref{thm:cws}. This observation yields the following corollary to Theorem~\ref{thm:cws}.

\begin{corollary}
Let $f\in\mathrm{C}_b(\C)$ and $k,m\in\N$. Under conditions of Theorem~\ref{thm:partition}, it holds that
\[
\lim_{N\to\infty} \e\left[\left(\int f\mathrm{d}\omega_{\boldsymbol\eta_N}\right)^k\left(\int \bar f\mathrm{d}\omega_{\boldsymbol\eta_N}\right)^m\right] = \left(\int f\mathrm{d}\omega_K\right)^k\left(\int \bar f\mathrm{d}\omega_K\right)^m.
\]
\end{corollary}

\section{Proofs}

\begin{proof}[Proof of Proposition~\ref{prop:minenergy}]
Let $\nu$ be such that $\dist(\nu,\omega_K)>0$. Without loss of generality we may assume that $\nu$ has finite energy. Recall that $g_K\equiv0$ on $K$. Thus, if $\supp(\nu)\subseteq K$, then $I_\ell[\nu]=I[\nu]>I[\omega_K]$ by the very definition of $\omega_K$. Otherwise, consider
\[
\widehat\nu := \nu_{|K}+\widehat\nu_{|K^c},
\]
where $\widehat\nu_{|K^c}$ is the balayage of $\nu_{|K^c}$ onto $K$ relative to $K^c$, \cite[Sec.~II.4]{SaffTotik}. Then it follows from \cite[Thm.~II.4.7]{SaffTotik} that $\widehat\nu$ has finite energy as well and $V^{\widehat\nu} < V^\nu + \int g_K\mathrm{d}\nu$
in $K^c$. Integrating both sides of this inequality against $\nu$ and using Fubini-Tonelli's theorem for the left-hand side, we get that
\[
\int V^\nu\mathrm{d}\widehat\nu < I[\nu] + \int g_K\mathrm{d}\nu.
\]
In fact, it also true that $V^{\widehat\nu} = V^\nu + \int g_K\mathrm{d}\nu$ everywhere on $K$ as $K$ is regular with respect to the Dirichlet problem, \cite[Sec.~II.4]{SaffTotik}. Therefore,
\[
I[\widehat\nu] < I[\nu] + 2\int g_K\mathrm{d}\nu = I_1[\nu] \leq I_\ell[\nu].
\]
The desired conclusion now follows from the fact that $\supp(\widehat\nu)\subseteq K$ and therefore $I[\omega_K]\leq I[\widehat\nu]$.
\end{proof}

Put, for brevity, $w_K:=e^{-g_K}$ and define
\begin{equation}
\label{deltaNK}
\delta_N(w_K) := \sup\left\{\prod_{m < n} |z_n - z_m|\prod_n w_K^{N-1}(z_n):~(z_1,\ldots,z_N)\in \C^N\right\}.
\end{equation}
\begin{lemma}
\label{lem:fekete}
Let $\{\lambda_1,\ldots,\lambda_N\}$ be any configuration satisfying
\[
\delta_N(w_K) = \prod_{m < n} |\lambda_n - \lambda_m | \prod_n w_K^{N-1}(\lambda_n),
\]
then $\{\lambda_1,\ldots,\lambda_N\}\subset K$.
\end{lemma}
\begin{proof}
Set
\[
P_{k,N}(z) := w_K^{N-1}(z)\prod_{n\neq k}(z-\lambda_n).
\]
Then $\|P_{k,N}\|_\C = |P_{k,N}(\lambda_k)|$ for each $k\in\{1,\ldots,N\}$ since
\[
\delta_N(w_K) = \|P_{k,N}\|_\C\prod_{m < n,~n,m\neq k} |\lambda_n - \lambda_m | \prod_{n\neq k} w_K^N(\lambda_n).
\]
As $\deg(P_{k,N})=N-1$, $w_K\equiv1$ on $K$, and $w_K(z)<1$ for $z\in K^c$, Bernstein-Walsh inequality \cite[Thm.~III.2.1]{SaffTotik} yields that
\[
|P_{k,N}(z)| < \|P_{k,N}\|_K \leq \|P_{k,N}\|_\C, \quad z\in K^c.
\] 
Hence, $P_{k,N}$ attains its maximum on $K$ and therefore $\lambda_k\in K$ for all $k\in\{1,\ldots,N\}$. 
\end{proof}

Let $E$ be a compact set satisfying all the conditions of Theorem~\ref{thm:partition}. Define
\begin{equation}
\label{Em}
E_m := \overline{\left\{z\in E\big|~\dist(z,\partial E)>\frac1m\right\}}, \quad m\in\N.
\end{equation}
\begin{lemma}
\label{lem:Em}
Let $E$ be as described and $\nu$ be a probability Borel measure supported on $E$ with finite energy. Then there exists a sequence of probability Borel measures $\{\nu_m\}$ such that $\supp(\nu_m)\subseteq E_m$, $\dist\left(\nu_m,\nu\right)\to0$ and $I[\nu_m]\to I[\nu]$ as $m\to\infty$. Moreover, if $\nu=\omega_E$, then $\nu_m$ can be taken $\omega_{E_m}$.
\end{lemma}
\begin{proof}
For each $m$ the measure $\nu$ can be written as
\[
\nu=\nu_{|E_m}+\nu_{|E^\circ\setminus E_m}+\nu_{|\partial E}.
\]
By the very definition of the sets $E_m$ it holds that $E_{m-1}\subset E_m$ and $E^\circ=\cup_m E_m$. Hence $\cap_m (E^\circ\setminus E_m)=\varnothing$ and therefore $|\nu_{|E^\circ\setminus E_m}|\to0$ as $m\to\infty$. If $\nu$ is supported only on the boundary of $E$, $\nu=\nu_{|\partial E}$, we set $\nu_m:=\widehat\nu_m$, where $\widehat\nu_m$ is the balayage of $\nu_{|\partial E}$ onto $E_m$ relative to $E_m^c$. Otherwise, $\nu_{|E_m}$ is not a zero measure for all $m$ large enough and therefore we can define for such $m$
\[
\nu_m := \alpha_m\nu_{|E_m} + \widehat\nu_m,
\]
where $\alpha_m:=1+|\nu_{|E^\circ\setminus E_m}|/|\nu_{|E_m}|$ is chosen so $|\nu_m|=1$, $\alpha_m\to1$ as $m\to\infty$. Clearly, $\supp(\nu_m)\subseteq E_m$.

The sequence $\{\nu_m\}$ has a weak$^*$ limit point, say $\nu^*$. Since $\nu^*(B)=\nu(B)$ for any compact set $B\subset E^\circ$, it holds that $\nu^*_{|E^\circ}=\nu_{|E^\circ}$ by the interior regularity of Borel measures. Therefore,
\[
\nu^*=\nu_{|E^\circ}+\nu_{|\partial E}^*,
\]
where $|\nu_{|\partial E}^*|=|\nu_{|\partial E}|=|\widehat\nu_m|$ and $\nu_{|\partial E}^*$ is a weak$^*$ limit point of $\{\widehat\nu_m\}$. Let us show that $\nu_{|\partial E}^*=\nu_{|\partial E}$. This will imply that $\nu^*=\nu$ and therefore $\nu_m\cws\nu$ as $\nu^*$ is an arbitrary weak$^*$ limit point of $\{\nu_m\}$.

Let $\Lambda\subseteq\N$ be a subsequence such that $\widehat\nu_m\cws\nu_{|\partial E}^*$ as $m\to\infty$, $m\in\Lambda$. Since $\supp(\nu_{|\partial E}^*)\subseteq\partial E$, it holds that
\[
V^{\nu_{|\partial E}^*}(z) = \lim_{m\to\infty,~m\in\Lambda} V^{\widehat\nu_m}(z), \quad z\in E^\circ.
\]
On the other hand, by the very properties of balayage \cite[Thm.~II.4.7]{SaffTotik}, it holds that
\[
V^{\widehat\nu_m}(z) = V^{\nu_{|\partial E}}(z) + c_m, \quad z\in E_m^\circ,
\]
where $c_m=\int g_{E_m}\mathrm{d}\nu_{|\partial E}$ and $g_{E_m}$ is the Green function for $E_m^c$ with pole at infinity. Since $E_{m+1}^c\subset E_m^c$, the maximum principle for harmonic functions applied to $g_{E_m}-g_{E_{m+1}}$ yields that the functions $g_{E_m}$ form a decreasing sequence on $\partial E$. Thus, $c:=\lim_{m\to\infty}c_m$ is well-defined. Hence, we have that
\[
V^{\nu_{|\partial E}^*} = V^{\nu_{|\partial E}} + c \quad \mbox{on} \quad E^\circ.
\]
Since logarithmic potentials are continuous in the fine topology by the very definition of the latter \cite[Sec.~I.5]{SaffTotik}, the equality above can be extended to every fine limit point of $E^\circ$. As $E$ is regular with respect to the Dirichlet problem in $E^c$, the fine closure of $E^\circ$ coincides with $E$ \cite[Thm.~A.2.1]{SaffTotik}. Thus,
\[
V^{\nu_{|\partial E}^*} = V^{\nu_{|\partial E}} + c \quad \mbox{on} \quad E\supseteq\big(\supp(\nu_{|\partial E})\cup\supp(\nu_{|\partial E}^*)\big).
\]
As both measures have finite energy and therefore are $C$-absolutely continuous, it holds that $\nu_{|\partial E}^*=\nu_{|\partial E}$ and $c=0$ by \cite[Thm.~II.4.6]{SaffTotik}.

To show that $I[\nu_m]\to I[\nu]$ as $m\to\infty$, it is enough to prove that $\limsup_{m\to\infty}I[\nu_m]\leq I[\nu]$ as the opposite inequality follows from the principle of descent \cite[Thm.~I.6.8]{SaffTotik}. It is a straightforward computation to get
\[
\nu = \nu_m - (\alpha_m-1)\nu_{|E^\circ} +\alpha_m\nu_{|E^\circ\setminus E_m} +\nu_{|\partial E}-\widehat\nu_m,
\]
where, again, the two middle terms containing $\alpha_m$ are not present when $\nu=\nu_{|\partial E}$. Thus, since $V^{\nu_{|\partial E}} - V^{\widehat\nu_m} \geq -c_m$ in $\C$ \cite[Thm.~II.4.7]{SaffTotik}, it holds that
\begin{eqnarray}
V^{\nu} &\geq& V^{\nu_m} - (\alpha_m-1)V^{\nu_{|E^\circ}} + \alpha_mV^{\nu_{|E^\circ\setminus E_m}} - c_m \nonumber \\
&\geq& V^{\nu_m} - (\alpha_m-1)V^{\nu_{|E^\circ}} -\alpha_m|\nu_{|E^\circ\setminus E_m}|\log\diam(E) - c_m \nonumber
\end{eqnarray}
on $E$, where we also used an observation that $V^{\sigma}+|\sigma|\log\diam(E)\geq0$ on $E$ for any positive measure $\sigma$ supported on $E$. Recall that $c_m,|\nu_{|E^\circ\setminus E_m}|\to0$ and $\alpha_m\to1$ (whenever present) as $m\to\infty$. Hence, integrating both sides of the last inequality against $\nu$ and taking the limit superior of the right-hand side yields
\[
I[\nu] \geq \limsup_{m\to\infty} \int V^{\nu_m}\mathrm{d}\nu = \limsup_{m\to\infty} \int V^\nu\mathrm{d}\nu_m.
\]
Using the estimate from below for $V^\nu$ once more with the same caveat concerning $\alpha_m$, we get that
\[
I[\nu] \geq \limsup_{m\to\infty}\left(I[\nu_m] - (\alpha_m-1)\int V^{\nu_{|E^\circ}}\mathrm{d}\nu_m\right).
\]
The desired inequality follows now from the fact that the integrals $\int V^{\nu_{|E^\circ}}\mathrm{d}\nu_m$ are uniformly bounded above. Indeed, notice that
\[
V^{\nu_{|E^\circ}} \leq V^\nu + \big(|\nu|-|\nu_{|E^\circ}|\big)\log\diam(E)
\]
and therefore we get by what precedes that
\begin{eqnarray}
\limsup_{m\to\infty}\int V^{\nu_{|E^\circ}}\mathrm{d}\nu_m & \leq & \limsup_{m\to\infty}\int V^\nu\mathrm{d}\nu_m + \big(|\nu|-|\nu_{|E^\circ}|\big)\log\diam(E) \nonumber \\
&\leq& I[\nu] + \big(|\nu|-|\nu_{|E^\circ}|\big)\log\diam(E). \nonumber
\end{eqnarray}

Finally assume that $\nu=\omega_E$. Since $\supp(\omega_E)\subseteq\partial E$, it holds that $\nu_m$ is simply the balayage of $\omega_E$ onto $E_m$ relative to $E_m^c$. Then $V^{\nu_m}=V^{\omega_E}+c_m=I[\omega_E]+c_m$ q.e. on $E_m$ \cite[Thm.~II.4.7]{SaffTotik}. However, the latter property uniquely characterizes equilibrium measures \cite[Thm.~I.3.3]{SaffTotik}.
\end{proof}

Let $\nu$ be a compactly supported probability Borel measure. Define
\begin{equation}
\label{omega-epsilon}
\mathrm{d}\nu_\varepsilon:=a_\varepsilon\mathrm{d}A, \quad a_\varepsilon(z) := \frac{1}{\pi\varepsilon^2}\int\mathbf{1}_\varepsilon(u-z)\mathrm{d}\nu(u),
\end{equation}
where $\mathbf{1}_\varepsilon$ is the indicator function of the disk $\{|u|<\varepsilon\}$. So defined, $\nu_\varepsilon$ is also a probability measure with compact support.
\begin{lemma}
\label{lem:energies1}
It holds that $\dist\left(\nu,\nu_\varepsilon\right)\leq\varepsilon$ and $I[\nu_\varepsilon] \to I[\nu]$ as $\varepsilon\to0$.
\end{lemma}
\begin{proof}
Let $f$ be a Lipschitz continuous function on $\supp(\nu_{\varepsilon})$ with Lipschitz constant 1 for some $\varepsilon>0$. Observe that $\supp(\nu)\subset\supp(\nu_\varepsilon)$. Then
\[
\int f\mathrm{d}\nu_\varepsilon - \int f\mathrm{d}\nu = \int \big(f(z)-f(u)\big)\frac{\mathbf{1}_\varepsilon(u-z)}{\pi\varepsilon^2}\mathrm{d}\nu(u)\mathrm{d}A(z) \leq \int |z-u|\frac{\mathbf{1}_\varepsilon(u-z)}{\pi\varepsilon^2}\mathrm{d}\nu(u)\mathrm{d}A(z) \leq \varepsilon
\]
by the Fubini-Tonelli theorem and since $|f(z)-f(u)|\leq|z-u|\leq\varepsilon$ for $z,u\in\supp(\nu_\varepsilon)$, $|z-u|\leq\varepsilon$. Hence, $\dist\left(\nu,\nu_\varepsilon\right)\leq\varepsilon$ and therefore $\nu_\varepsilon\cws\nu$ as $\varepsilon\to0$.

If $I[\nu]=\infty$, then
\[
I_M[\nu] := \int \min\left\{M,\log\frac{1}{|z-u|}\right\}\mathrm{d}\nu^{\otimes2}(z,u)
\]
diverges to infinity as $M\to\infty$ by the monotone convergence theorem. Since the functional $I_M$ is defined with respect to a continuous kernel, it follows from the weak$^*$ convergence of measures that
\[
\liminf_{\varepsilon\to0}I[\nu_\varepsilon] \geq \lim_{\varepsilon\to0} I_M[\nu_\varepsilon] = I_M[\nu].
\]
Because the inequality above is true for any $M$, the limit of $I[\nu_\varepsilon]$ as $\varepsilon\to0$ diverges to infinity as well. The rest of the lemma is the content of \cite[Sec.~3.2]{uHedMak}.
\end{proof}

For further use, let us state the following trivial modification of the principle of descent \cite[Thm.~I.6.8]{SaffTotik} for empirical measures.
\begin{lemma}
\label{lem:descent}
Let $\{\omega_{\boldsymbol\eta_N}\}$ be a sequence of empirical measures that converges weak$^*$ to some probability Borel measure $\nu$. Then
\[
I[\nu] \leq \liminf_{N\to\infty} I_*[\omega_{\boldsymbol\eta_N}], \quad I_*[\omega_{\boldsymbol\eta_N}] = \int_{z\neq u}\log\frac{1}{|z-u|}\mathrm{d}\omega_{\boldsymbol\eta_N}^{\otimes2}(z,u).
\]
\end{lemma}
\begin{proof}
Since $\omega_{\boldsymbol\eta_N}\cws\omega$, it holds that $\omega_{\boldsymbol\eta_N}^{\otimes2}\cws\nu^{\otimes2}$. Hence, it follows from the monotone convergence theorem and the continuity of $\min\left\{M,\log\frac{1}{|z-u|}\right\}$ on $\C^2$ that
\begin{eqnarray}
I[\nu] &=& \lim_{M\to\infty} \int \min\left\{M,\log\frac{1}{|z-u|}\right\}\mathrm{d}\nu^{\otimes2}(z,u) \nonumber \\
&=& \lim_{M\to\infty}\lim_{N\to\infty} \int \min\left\{M,\log\frac{1}{|z-u|}\right\}\mathrm{d}\omega_{\boldsymbol\eta_N}^{\otimes2}(z,u) \nonumber \\
& \leq & \lim_{M\to\infty}\liminf_{N\to\infty} \left(I_*[\omega_{\boldsymbol\eta_N}]+\frac MN\right) = \liminf_{N\to\infty} I_*[\omega_{\boldsymbol\eta_N}]. \nonumber
\end{eqnarray}
\end{proof}

\begin{lemma}
\label{lem:points}
Let $\nu_\varepsilon$ be as in \eqref{omega-epsilon}. Then there exist configurations $\boldsymbol\eta^{N,\varepsilon}$ such that $\min_{j\neq k}|\eta_j^{N,\varepsilon}-\eta_k^{N,\varepsilon}|\geq c^\prime/\sqrt N$, $\dist\left(\nu_\varepsilon,\omega_{\boldsymbol\eta^{N,\varepsilon}}\right) = \mathcal{O}\left(N^{-1/4}\right)$, and $I_*[\omega_{\boldsymbol\eta^{N,\varepsilon}}]\to I[\nu_\varepsilon]$ as $N\to\infty$, for some constant $c^\prime$ that depends on $\nu_\epsilon$ but does not depend on $N$.
\end{lemma}
\begin{proof}
For convenience, set $M:=\left\lceil N^{1/2}\right\rceil$. Since $\nu_\varepsilon$ is absolutely continuous with respect to $\mathrm{d}A$, there exist real numbers $x_1<x_2<\cdots<x_M$ and a positive constant $b_M$ such that the vertical strips 
\[
S_j:=\left\{z:~x_j\leq\re(z)<x_j+\frac{b_M}{M}\right\}, \quad j\in\{1,\ldots,M\},
\]
are mutually disjoint and $\supp(\nu_\varepsilon)\subset\cup_{j=1}^MS_j$. Clearly, $b_M$ is bounded above by $\diam(\supp(\nu_\varepsilon))$ and below by the 1-dimensional Lebesgue measure of the projection of $\supp(\nu_\varepsilon)$ on the $x$-axis. Moreover, it holds by \eqref{omega-epsilon} that
\[
\nu_\varepsilon\left(S_j\right) = \int_{S_j}\mathrm{d}\nu_\varepsilon \leq \frac{c_M}{\pi\varepsilon^2}\frac{\diam(\supp(\nu_\varepsilon))}{M}.
\]
In particular, the inequality above implies that
\[
M_j:= \left\lfloor \nu_\varepsilon\left(S_j\right)N\right\rfloor\leq c_1N^{1/2}
\]
(all the constants $c_i$ appearing from now on depend only $\nu_\varepsilon$). Furthermore, as $\nu_\varepsilon$ is a unit measure, it holds that
\[
N - M < \sum_{j=1}^MM_j \leq N.
\]
Let now $y_{j,1}$ be the greatest ordinate such that $\nu_\varepsilon\left(S_j\cap\left\{z:~y_{j,1}>\im(z)\right\}\right) = 0$. Then for each $k\in\left\{1,\ldots,M_j\right\}$ there exists $y_{j,k+1}$ such that
\[
\nu_\varepsilon\left(R_{j,k}\right)=N^{-1}, \quad R_{j,k} := S_j\cap\left\{z:~y_{j,k}\leq\im(z) < y_{j,k+1}\right\}.
\]
Observe that by \eqref{omega-epsilon} it holds that
\[
y_{j,k+1}-y_{j,k} \geq \frac{\pi\varepsilon^2}{b_M}\frac MN \geq \frac{\pi\varepsilon^2}{b_M}\frac1{N^{1/2}} \geq \frac{c_2}{N^{1/2}}.
\]

Set
\[
\eta_{j,k}:= x_j + \mathrm{i}y_{j,k}, \quad k\in\{1,\ldots,M_j+1\} \quad \mbox{and} \quad j\in\{1,\ldots,M\}.
\]
Then the collection $\boldsymbol\eta:=\left\{\eta_{j,k}\right\}$ contains at least $N$ and at most $N+M$ points and by the very construction
\[
\min_{(j,k)\neq (i,l)}\left|\eta_{j,k}-\eta_{i,l}\right| \geq \frac{\min\{b_M,c_2\}}{N^{1/2}} \geq \frac{c^\prime}{N^{1/2}}
\]
for some $c^\prime>0$ that depends only on $\nu_\varepsilon$ as the constants $b_M$ are uniformly bounded above and below. 

Let $R:=\cup_{j=1}^M\cup_{k=1}^{M_j}R_{j,k}$. Then
\[
I[\nu_{\varepsilon|R}] = \left(\sum_j\sum_{k,l}\int_{R_{j,k}\times R_{j,l}}+ 2\sum_{j<i}\sum_{k,l}\int_{R_{j,k}\times R_{i,l}}\right)\log\frac1{|z-u|}\mathrm{d}\nu_\varepsilon^{\otimes2}(z,u) =: \sum_j I_j + 2\sum_{j<i} I_{j,i}.
\]
To estimate the first sum notice that the monotonicity of the logarithm and the choice of the rectangles $R_{j,k}$ yield
\[
I_j \geq \frac1{N^2}\sum_{k=1}^{M_j}\sum_{l=1}^{M_j}\log\frac1{\diam(R_{j,k}\cup R_{j,l})}.
\]
Moreover, by the very choice of the points $\eta_{j,k}$, we have that $\big|\eta_{j,k}-\eta_{j,l}\big| \geq c_2|k-l|N^{-1/2}$. Thus, as the width of each $R_{j,k}$ is $b_M/M$, we deduce that
\[
\diam\left(R_{j,k}\cup R_{j,l}\right) \leq \left\{
\begin{array}{ll}
\left|\eta_{j,k}-\eta_{j,l+1}\right|\sqrt{1 + \frac {c_3}{(l+1-k)^2}}, & l\geq k, \smallskip \\
\left|\eta_{j,k+1}-\eta_{j,l}\right|\sqrt{1 + \frac {c_3}{(k+1-l)^2}}, & l<k.
\end{array}
\right.
\]
Hence, using the fact that
\[
-\frac1{2N^2}\sum_{k=1}^{M_j}\sum_{l=k}^{M_j}\log\left(1+ \frac {c_3}{(l+1-k)^2}\right) \geq -\frac1{2N^2}\sum_{k=1}^{M_j}\sum_{l=k}^{M_j}\frac {c_3}{(l+1-k)^2} \geq  -\frac{c_4}{2}\frac{M_j}{N^2}
\]
and that exactly the same estimate holds for $l\in\{1,\ldots,k-1\}$, we get that
\[
I_j \geq \frac1{N^2}\sum_{k=1}^{M_j+1}\sum_{l\neq k,l=1}^{M_j+1}\log\frac1{|\eta_{j,k}-\eta_{j,l}|} - \frac1{N^2}\sum_{k=1}^{M_j}\log\frac1{|\eta_{j,k}-\eta_{j,k+1}|} - c_4\frac{M_j}{N^2}
\]
and respectively that
\[
\sum_{j=1}^M I_j \geq \frac1{N^2}\sum_{j=1}^M\sum_{k\neq l,~k,l=1}^{M_j+1}\log\frac1{|\eta_{j,k}-\eta_{j,l}|} - c_5\frac{\log N}{N}.
\]
To estimate the second sum, fix $j,i\in\{1,\ldots,M\}$ such that $j<i$. It can be readily observed that for each $k\in\left\{1,\ldots,M_j\right\}$ there exists $l_k\in\left\{1,\ldots,M_i\right\}$ such that
\[
\diam\left(R_{j,k}\cup R_{i,l}\right) = \left\{
\begin{array}{ll}
|\eta_{j,k+1}-\eta_{i,l}+b_M/M|, & l<l_k, \smallskip \\
|\eta_{j,k}-\eta_{i,l+1}+b_M/M|, & l\geq l_k.
\end{array}
\right.
\]
Therefore, it holds that
\[
\diam\left(R_{j,k}\cup R_{i,l}\right) \leq \left(1+\frac{1}{i-j}\right)\left\{
\begin{array}{ll}
|\eta_{j,k+1}-\eta_{i,l}|, & l<l_k, \smallskip \\
|\eta_{j,k}-\eta_{i,l+1}|, & l\geq l_k.
\end{array}
\right.
\]
By the above inequality and since $M_j\leq c_1N^{1/2}$, we deduce that
\[
I_{j,i} \geq \frac1{N^2}\sum_{k=1}^{M_j}\left(\sum_{l=1}^{l_k-1}\log\frac1{|\eta_{j,k+1}-\eta_{i,l}|} + \sum_{l=l_k}^{M_i}\log\frac1{|\eta_{j,k}-\eta_{i,l+1}|}\right) - \frac{1}{i-j}\frac{c_1^2}{N}.
\]
It can be readily verified that the sequence $\{l_k\}_{k=1}^{M_j}$ is non-decreasing and therefore the pairs of indices in the double sums above never repeat themselves. Thus,
\[
I_{j,i} \geq \frac1{N^2}\sum_{k=1}^{M_j+1}\sum_{l=1}^{M_i+1}\log\frac1{|\eta_{j,k}-\eta_{i,l}|} - \frac1{N^2}\sum_{m=1}^{M_j+M_i+1}\log\frac1{|\eta_{j,k_m}-\eta_{i,l_m}|} - \frac{1}{i-j}\frac{c_1^2}{N}
\]
for some sequence $\big\{(k_m,l_m)\big\}$. So, using the fact that $|\eta_{j,k_m}-\eta_{i,l_m}|\geq c^\prime N^{-1/2}$, we get that
\begin{eqnarray}
2\sum_{j<i}I_{j,i} &\geq& 2\sum_{j<i}\left(\sum_{k,l}\log\frac1{|\eta_{j,k}-\eta_{i,l}|} - c_6\frac{\log N}{N^{3/2}} - \frac{1}{i-j}\frac{c_1^2}{N}\right) \nonumber \\
&\geq& 2\sum_{j<i}\sum_{k,l}\log\frac1{|\eta_{j,k}-\eta_{i,l}|} - c_7\frac{\log N}{N^{1/2}}. \nonumber
\end{eqnarray}

Combing the estimates for $\sum I_j$ and $\sum I_{j,i}$, we derive that
\[
I[\nu_\varepsilon] = I[\nu_{\varepsilon|R}] + \int V^{\nu_\varepsilon+\nu_{\varepsilon|R}}\mathrm{d}\nu_{\varepsilon|R^c} \geq \limsup_{N\to\infty}\left(I_*[\omega_{\boldsymbol\eta}] - c_7\frac{\log N}{N^{1/2}}+\int V^{\nu_\varepsilon+\nu_{\varepsilon|R}}\mathrm{d}\nu_{\varepsilon|R^c}\right) = \limsup_{N\to\infty}I_*[\omega_{\boldsymbol\eta}],
\]
where the last equality follows from the fact that $\nu_{\varepsilon|R^{c}}\leq N^{-1/2}$ and since the potentials $V^{\nu_\varepsilon+\nu_{\varepsilon|R}}$ are uniformly bounded on $\supp(\nu_\varepsilon)$ as the measures $\nu_\varepsilon+\nu_{\varepsilon|R}$ are absolutely continuous with respect to $\mathrm{d}A$ with uniformly bounded densities.

To show that
\[
I[\nu_\varepsilon] \leq \liminf_{N\to\infty}I_*[\omega_{\boldsymbol\eta}],
\]
we appeal to Lemma~\ref{lem:descent}, which asserts the above inequality granted we show that $\omega_{\boldsymbol\eta}\cws\nu_\varepsilon$ as $N\to\infty$. To this end, observe that the differences $y_{j,k+1}-y_{j,k}$ cannot be too large to often. Indeed, let $B_N$ be the number of the differences $y_{j,k+1}-y_{j,k}$ that are large than $N^{-1/4}$. Since the sum of the heights of the rectangle $R_{j,k}$ for a fixed index $j$ is bounded by $\diam(\supp(\nu_\varepsilon))$ independently of $j$, it holds that
\[
\frac{B_N}{N^{1/4}}+c_2\frac{N-B_N}{N^{1/2}} \leq \diam(\supp(\nu_\varepsilon))M,
\]
which immediately implies that $B_N=\mathcal{O}\left(N^{3/4}\right)$. Hence, if $f$ is a function that is bounded by 1 in modulus and satisfies the Lipschitz condition with constant 1 on $\supp(\nu_\varepsilon)$, then we can write
\[
\int f\mathrm{d}\left(\nu_\varepsilon - \omega_{\boldsymbol\eta^{N,\varepsilon}}\right)  = \sum_{j=1}^M\sum_{k=1}^{M_j}\int\left[f-f\left(\eta_{j,k}\right)\right]\mathrm{d}\nu_\varepsilon + \int_{R^c} f\mathrm{d}\nu_\varepsilon + \frac1N\sum_{j=1}^{M}f\left(\eta_{j,M_j+1}\right).
\]
Observe that each of the last two terms on the right-hand side of the equality above is of order $N^{-1/2}$ as $|f|\leq1$, $\nu_\varepsilon(R^c)<N^{-1/2}$, and $M=\big\lceil N^{1/2}\big\rceil$. Further, split the double sum into two: one over those rectangles that have heights at most $N^{-1/4}$ and the rest of them. Since the number of the ``large'' rectangles is $B_N=\mathcal{O}\left(N^{3/4}\right)$ and $|f|\leq1$, we get that this part of the sum is of order $\mathcal{O}\left(N^{-1/4}\right)$. On another hand, using the Lipschitz continuity of $f$ on the first group of the rectangles and since the width of each rectangle is $b_M/M$, we deduce that
\[
\int_{R_{j,k}}\left|f-f\left(\eta_{j,k}^{N,\varepsilon}\right)\right|\mathrm{d}\nu_\varepsilon = \mathcal{O}\left(N^{-5/4}\right)
\]
and therefore the first part of the sum is also of order $\mathcal{O}\left(N^{-1/4}\right)$. Hence, $\dist\left(\nu_\varepsilon,\omega_{\boldsymbol\eta^{N,\varepsilon}}\right) = \mathcal{O}\left(N^{-1/4}\right)$ and the claim follows.

Finally, we define $\boldsymbol\eta^{N,\varepsilon}:=\left\{\eta_j^{N,\varepsilon}\right\}_{j=1}^N$ by selecting any $N$ point from the collection $\boldsymbol\eta$. As we are discarding at most $N^{1/2}$ points, $\omega_{\boldsymbol\eta^{N,\varepsilon}}$ still converges to $\nu_\varepsilon$ in the weak$^*$ topology. Moreover, the behavior of the discrete energies also remains unaltered as the absolute value of the  contribution of the removed points is of order $\mathcal{O}\left(N^{-1/2}\log N\right)$.
\end{proof}

\begin{lemma}
\label{lem:nearpoints}
Let $\nu_\varepsilon$ be defined by \eqref{omega-epsilon}, $\boldsymbol\eta^{N,\varepsilon}$ be as in Lemma~\ref{lem:points}, and
\[
O_{N,\varepsilon} := \left\{z_1~|~\big|z_1-\eta^{N,\epsilon}_1\big|<\frac{c^\prime}{3\sqrt N}\right\} \times \cdots \times \left\{z_N~|~\big|z_N-\eta^{N,\epsilon}_N\big|<\frac{c^\prime}{3\sqrt N}\right\}.
\]
 Then for all\footnote{For brevity, we slightly abuse the notation and assume that $\boldsymbol\eta\in O_{N,\varepsilon}$ stands for $\boldsymbol\eta=\{z_1,\ldots,z_N\}$, where $(z_1,\ldots,z_N)\in O_{N,\varepsilon}$.} $\boldsymbol\eta\in O_{N,\varepsilon}$, it holds that
\[
\dist\left(\nu_\varepsilon,\omega_{\boldsymbol\eta}\right)\leq \mathcal{O}\left(N^{-1/4}\right) \quad \mbox{and} \quad \left|I_*[\omega_{\boldsymbol\eta}]-I_*[\omega_{\boldsymbol\eta^{N,\varepsilon}}]\right|\leq c^{\prime\prime}N^{-1/2}\log N
\]
for some constant $c^{\prime\prime}$ that depends only on $\nu_\varepsilon$.
\end{lemma}
\begin{proof}
It holds for every $\boldsymbol\eta\in O_{N,\varepsilon}$ that
\[
\dist\left(\omega_{\boldsymbol\eta},\omega_{\boldsymbol\eta^{N,\varepsilon}}\right) = \sup_f\left|\int f\mathrm{d}\omega_{\boldsymbol\eta} - \int f\mathrm{d}\omega_{\boldsymbol\eta^{N,\varepsilon}}\right| \leq \frac1N\sum_{j=1}^N|\eta_j-\eta_j^{N,\varepsilon}| <  \frac{c^\prime}{3\sqrt N}
\]
by Lipschitz continuity of $f$. Thus, the first claim follows from the triangle inequality. Furthermore, it holds that
\[
\left|I_*[\omega_{\boldsymbol\eta}]-I_*[\omega_{\boldsymbol\eta^{N,\varepsilon}}]\right| \leq \frac1{N^2}\sum_{n\neq m}\big|\log\left|1 + \frac{(\eta_n-\eta_n^{N,\varepsilon})-(\eta_m-\eta_m^{N,\varepsilon})}{\eta_n^{N,\varepsilon}-\eta_m^{N,\varepsilon}}\right|\big| \leq \frac{c^\prime}{N^{5/2}}\sum_{n\neq m}\frac{1}{|\eta_n^{N,\varepsilon}-\eta_m^{N,\varepsilon}|}.
\]
Replacing the collection $\left\{\eta_j^{N,\varepsilon}\right\}$ by $\left\{\eta_{j,k}\right\}$ constructed in Lemma~\ref{lem:points}, we shall only increase the sum on the right-hand side of the above chain of inequalities. As
\[
\sum_{l\neq k,l=1}^{M_j}\frac{1}{|\eta_{j,k}-\eta_{j,l}|} + \sum_{i\neq j,i=1}^{M}\sum_{l=1}^{M_i}\frac{1}{|\eta_{j,k}-\eta_{i,l}|} \leq c_1 \sum_{i=1}^M N^{1/2}\log N \leq c_2 N\log N
\]
for a fixed point $\eta_{j,k}$, the desired result follows.
\end{proof}

\begin{proof}[Proof of Theorem~\ref{thm:partition}] We start by proving the upper bound in \eqref{partitionlimit}. By Lemma~\ref{lem:fekete}, each extremal configuration $\{\lambda_1,\ldots,\lambda_N\}$ realizing the supremum in \eqref{deltaNK} is, in fact, a configuration of Fekete points for $K$ since $w_K\equiv1$ on $K$. That is,
\[
\delta_N(w_K) = \delta_N(K) := \sup\left\{\prod_{m < n} |z_n - z_m|:~(z_1,\ldots,z_N)\in K^N\right\}.
\]
Hence, we get from \eqref{ZNsbeta} that
\[
Z_{N,s,\beta} \leq \delta_N^\beta(K)\left(\int_\C w_K^{(s+1-N)\beta} \mathrm{d}A\right)^N.
\]
By \eqref{snbeta} and since $w_K(z)\sim|z|^{-1}$ as $|z|\to\infty$, it holds that
\[
\limsup_{N\to\infty}N^{-1}\log\left(\int_\C w_K^{(s+1-N)\beta} \mathrm{d}A\right) \leq \lim_{N\to\infty}N^{-1}\log\left(\int_\C w_K^{2+c_0} \mathrm{d}A\right) = 0.
\]
In other words,
\[
\limsup_{N\to\infty}N^{-2}\log Z_{N,s,\beta} \leq \beta\lim_{N\to\infty}N^{-2}\log \delta_N(K) = -\frac\beta2I[\omega_K],
\]
where the last equality is a well known Fekete-Szeg\H{o} theorem \cite[Thm.~5.5.2]{Ransford}.

To prove the lower bound in \eqref{partitionlimit}, let $K_m$ be defined by \eqref{Em}. Further, let $\omega_{m,\varepsilon}$ be the measure defined by \eqref{omega-epsilon} for $\nu=\omega_{K_m}$ and any $\varepsilon\in(0,\frac1m)$. Clearly, $\supp(\omega_{m,\varepsilon})\subset K$. Then we get from \eqref{ZNsbeta} that
\begin{eqnarray}
Z_{N,s,\beta}  &\geq&  \int_{\supp(\omega_{m,\varepsilon})^N}  \prod_{m < n} |z_n - z_m|^\beta \mathrm{d}A^{\otimes N}(z_1,\ldots,z_N) \nonumber \\
&=& \int \prod_{m < n} |z_n - z_m|^\beta \, \exp\left\{- \sum_{n=1}^N\log a_{m,\varepsilon}(z_n)\right\}  \, \mathrm{d}\omega_{m,\varepsilon}^{\otimes N}(z_1,\ldots,z_N).  \nonumber
\end{eqnarray}
Hence, it follows from Jensen's inequality that
\begin{eqnarray}
\log Z_{N,s,\beta} &\geq& -\int\left(\beta\sum_{m<n}\log\frac{1}{|z_n-z_m|} + \sum_{n=1}^N\log a_{m,\varepsilon}(z_n)\right)\mathrm{d}\omega_{m,\varepsilon}^{\otimes N}(z_1,\ldots,z_N)  \nonumber \\
& = & -\beta\frac{N(N-1)}{2}I[\omega_{m,\varepsilon}] - N\int\log a_{m,\varepsilon}\mathrm{d}\omega_{m,\varepsilon}. \nonumber
\end{eqnarray}
As $a_{m,\varepsilon}\leq1/\pi\varepsilon^2$, the integral $\int\log a_{m,\varepsilon}\mathrm{d}\omega_{m,\varepsilon}$ is finite and therefore
\[
\liminf_{N\to\infty}\frac1{N^2}\log Z_{N,s,\beta} \geq -\frac\beta2\lim_{\varepsilon\to0}I[\omega_{m,\varepsilon}] = -\frac\beta2I[\omega_{K_m}],
\] 
where we used Lemma~\ref{lem:energies1}.  Thus, the lower bound in \eqref{partitionlimit} follows from Lemma~\ref{lem:Em}.

For brevity, put
\[
P_N(\nu,\epsilon) := \prob\left\{\dist\left(\nu,\omega_{\boldsymbol\eta}\right)<\epsilon\right\}, \quad \boldsymbol\eta=\{\eta_1,\ldots,\eta_N\}.
\]
Further, define
\[
O_{\nu,\epsilon} := \bigg\{(\eta_1,\ldots,\eta_N)\in\C^N~|~\dist\left(\nu,\omega_{\boldsymbol\eta}\right)<\epsilon,\quad\boldsymbol\eta=\{\eta_1,\ldots,\eta_N\}\bigg\}.
\]
Observe that $O_{\nu,\epsilon}$ is an open subset of $\C^N$ for any $\nu$ and $\epsilon>0$. 

To prove the upper bound in \eqref{largedeviation}, let $\boldsymbol\alpha^{N,\epsilon}$ be a configuration that maximizes
\[
\exp\left\{ -(\beta s-2-c_0)\sum_{n=1}^N g_K(z_n) \right\} \,\prod_{m < n} |z_n - z_m|^\beta
\]
among all $N$-point configurations $\boldsymbol\alpha$ satisfying $\dist\left(\nu,\omega_{\boldsymbol\alpha}\right)\leq\epsilon$. Such a configuration always exists since the function above decays as $z^{-(\beta(s-N+1)-2-c_0)}$ by \eqref{snbeta} in each coordinate. Therefore we are simply looking for a place where a continuous function reaches its maximum on some sufficiently large ball in $\C^N$. Then
\begin{eqnarray}
\log P_N(\nu,\epsilon) &=& \log\left(\int_{O_{\nu,\epsilon}}\Omega_{N,s,\beta}(z_1,\ldots,z_N)\mathrm{d}A^{\otimes N}(z_1,\ldots,z_N)\right)\nonumber\\
&\leq& - \log Z_{N,s,\beta} - (\beta s-2-c_0)N\int g_K\mathrm{d}\omega_{\boldsymbol\alpha^{N,\epsilon}}   - \frac\beta2N^2I_*[\omega_{\boldsymbol\alpha^{N,\epsilon}}] + N\log\int_\C e^{-(2+c_0)g_K}\mathrm{d}A.\nonumber
\end{eqnarray}
Let $\nu^\epsilon$ be a weak$^*$ limit point of $\{\omega_{\boldsymbol\alpha^{N,\epsilon}}\}$ (again, as all the configurations $\boldsymbol\alpha^{N,\epsilon}$ belong to a ball of fixed radius in $\C^N$, the measures $\omega_{\boldsymbol\alpha^{N,\epsilon}}$ are compactly supported). Then, along the subsequence for which $\omega_{\boldsymbol\alpha^{N,\epsilon}}\cws\nu^\epsilon$, it holds that
\[
\limsup_{N\to\infty}N^{-2}\log P_N(\nu,\epsilon) \leq \frac\beta2\bigg(I[\omega_K]-I_\ell[\nu^\epsilon]\bigg)
\]
by Lemma~\ref{lem:descent}, limit \eqref{partitionlimit}, and since $\int g_K\mathrm{d}\omega_{\boldsymbol\alpha^{N,\epsilon}}\to\int g_K\mathrm{d}\nu^\epsilon$. Let now $\{\nu^{\epsilon_M}\}_{M=1}^\infty$, $\epsilon_M\to0$ as $M\to\infty$, be an arbitrary sequence of weak$^*$ limit points constructed above. Since $\dist\left(\nu^{\epsilon_M},\nu\right)\leq\epsilon_M$, it holds that $\nu^{\epsilon_M}\cws\nu$ as $M\to\infty$. Thus, the principle of descent and continuity of $g_K$ yield that $\liminf_{M\to\infty}I_\ell[\nu^{\epsilon_M}]\geq I_\ell[\nu]$. This includes the case $\supp(\nu)\cap K^c\neq\varnothing$ and $\ell=0$. In this situation $\supp(\nu^{\epsilon_M})\cap K^c\neq\varnothing$ for all large $M$ and $I_0[\nu]=I_0[\nu^{\epsilon_M}]=\infty$. Therefore,
\[
\limsup_{\epsilon\to0}\limsup_{N\to\infty}N^{-2}\log P_N(\nu,\epsilon) \leq \frac\beta2\bigg(I[\omega_K]-I_\ell[\nu]\bigg).
\]

Observe also that the above considerations imply the existence of the full limit in \eqref{largedeviation} when $I[\nu]=\infty$ or when $\ell=0$ and $\supp(\nu)\cap K^c\neq\varnothing$ (clearly, the limit is $-\infty$). Thus, for the lower bound in \eqref{largedeviation}, it is enough to consider measures with finite energy and, when $\ell=0$, only those that are supported on $K$.

Assume first that $\ell>0$. Let $\nu$ be a Borel probability measure with finite energy and $\nu_\epsilon$ be the measure defined by \eqref{omega-epsilon} for some $\epsilon>0$. Let further $\{\boldsymbol\eta^{N,\epsilon}\}$ be a sequence of configurations constructed in Lemma~\ref{lem:points} and $O_{N,\epsilon}$ be neighborhood constructed in Lemma~\ref{lem:nearpoints}. It follows immediately from Lemma~\ref{lem:nearpoints}, the triangle inequality, and Lemma~\ref{lem:energies1} that $O_{N,\epsilon}\subset O_{\nu_\epsilon,\epsilon}\subset O_{\nu,2\epsilon}$. Hence,
\[
P_N(\nu,2\epsilon) \geq \frac1{Z_{N,s,\beta}}\int_{O_{N,\epsilon}} \exp\left\{ -\beta s\sum_{n=1}^N g_K(z_n) \right\} \,\prod_{m < n} |z_n - z_m|^\beta\mathrm{d}A^{\otimes N}(z_1,\ldots,z_N). 
\]
Then we deduce from Jensen's inequality that
\begin{eqnarray}
\log P_N(\nu,2\epsilon) &\geq& -\log Z_{N,s,\beta} + \log|O_{N,\epsilon}| \nonumber \\
&&- \frac\beta2\int_{O_{N,\epsilon}}\left(N^2I_*[\omega_{\boldsymbol\eta}]+2sN\int g_K\mathrm{d}\omega_{\boldsymbol\eta}\right)\frac{\mathrm{d}A^{\otimes N}(\eta_1,\ldots,\eta_N)}{|O_{N,\epsilon}|} \nonumber\\
&\geq& -\log Z_{N,s,\beta} - \frac\beta2\left(N^2I_*[\omega_{\boldsymbol\eta^{N,\epsilon}}]+2sN\int g_K\mathrm{d}\nu_\epsilon\right) + N\log\left(4(c^\prime)^2\pi\right)  \nonumber \\
&& - 9N \log N - \frac\beta2\bigg(c^{\prime\prime}N\log N+2sN\left(\max_{O_{N,\epsilon}}\int g_K\mathrm{d}\omega_{\boldsymbol\eta} - \int g_K\mathrm{d}\nu_\epsilon\right)\bigg) \nonumber 
\end{eqnarray}
by Lemma~\ref{lem:nearpoints} and since $\mathrm{d}A^{\otimes N}/|O_{N,\epsilon}|$ is a probability measure on $O_{N,\epsilon}$, where $|O_{N,\epsilon}|=\left(4(c^\prime)^2\pi/9N\right)^N$ is the volume of $O_{N,\epsilon}$. By compactness of $\overline{O_{N,\epsilon}}$ it holds that
\[
\max_{O_{N,\epsilon}}\int g_K\mathrm{d}\omega_{\boldsymbol\eta} = \int g_K\mathrm{d}\omega_{\boldsymbol\eta_N}
\]
for some $\boldsymbol\eta_N=\left\{\eta_1^N,\ldots,\eta_N^N\right\}$ such that $\left(\eta_1^N,\ldots,\eta_N^N\right)\in\overline{O_{N,\epsilon}}$. By Lemma~\ref{lem:nearpoints}, it holds that $\omega_{\boldsymbol\eta_N}\cws\nu_\varepsilon$ as $N\to\infty$ and therefore
\[
\max_{O_{N,\epsilon}}\int g_K\mathrm{d}\omega_{\boldsymbol\eta} - \int g_K\mathrm{d}\nu_\epsilon \to 0 \quad \mbox{as} \quad N\to\infty.
\]
Hence,
\[
\liminf_{N\to\infty}N^{-2}\log P_N(\nu,2\epsilon) \geq \frac\beta2\bigg(I[\omega_K]-I_\ell[\nu_\epsilon]\bigg)
\]
by \eqref{partitionlimit} and Lemma~\ref{lem:points}. The lower bound in \eqref{largedeviation} follows now from Lemma~\ref{lem:energies1}.

Finally, assume that $\ell=0$. Let $\nu$ be a probability Borel measure with finite energy supported on~$K$ and $\{\nu_m\}$ be a sequence of measures granted by Lemma~\ref{lem:Em}. Then $(\nu_m)_\epsilon$ is supported in $K^\circ$ for every $\epsilon<1/2m$ and so are the measures $\omega_{\boldsymbol\eta^{N,\epsilon}}$ and $\omega_{\boldsymbol\eta}$ constructed in Lemmas~\ref{lem:points} and~\ref{lem:nearpoints} for $(\nu_m)_\epsilon$. Thus, the argument we used for the case $\ell>0$ yields now that
\[
\liminf_{\epsilon\to0}\liminf_{N\to\infty}N^{-2}P_N(\nu_m,\epsilon) \geq \frac\beta2\bigg(I[\omega_K]-I_\ell[\nu_m]\bigg).
\]
To show the lower bound in \eqref{largedeviation} it remains only to observe that $P_N(\nu,\epsilon+\epsilon_m)\geq P_N(\nu_m,\epsilon)$, where $\epsilon_m=\dist\big(\nu,\nu_m\big)$, and $I[\nu_m]\to I[\nu]$ as $m\to\infty$ by Lemma~\ref{lem:Em}.
\end{proof}

\begin{proof}[Proof of Proposition~\ref{prop:beta2}]
Let $\{\pi_{n,s}\}_{n=0}^{N-1}$ be the sequence of orthonormal polynomials with respect to the inner product
\[
\langle f,g\rangle = \int f\overline g\exp\{-2sg_K\}\mathrm{d}A.
\]
Denote by $\varkappa_{n,s}$ the leading coefficient of $\pi_{n,s}$. Then it holds \cite[Sec.~5.4]{Deift} that
\[
\log Z_{N,s,2} = \log N!-2\sum_{n=0}^{N-1}\log\varkappa_{n,s}.
\]
It was shown in \cite[Thm.~1]{SinYa12} that $K$ as described 
\[
\varkappa_{n,s} = \frac{1}{\cp^{n+1}(K)} \sqrt{\frac{n+1}{\pi}\left(1-\frac{n+1}{s}\right)} \left[1+\mathcal{O}\left(\frac{1}{n^{2\alpha}}\right)\right].
\]
Hence, it follows from \eqref{capacity} that
\[
\log Z_{N,\infty,2} = -N(N+1)I[\omega_K] +N\log\pi + \mathcal{O}(1)
\]
and more generally
\[
\log Z_{N,s,2} = \log Z_{N,\infty,2} - \sum_{k=0}^{N-1}\log(s-k) + N\log s + \mathcal{O}(1).
\]
It follows from the Stirling's formula $\log n!=\big(n+1/2\big)\log n - n + \mathcal{O}(1)$ that
\[
-\sum_{k=0}^{N-1}\log(s-k) + N\log s = N + (s-N)\log\big(1-s^{-1}N\big)+ \mathcal{O}(\log N)
\]
which yields \eqref{beta2}.
\end{proof}

For $\varepsilon>0$, set
\begin{equation}
\label{UNepsilon}
U_{N,\varepsilon}:=\left\{(z_1,\ldots,z_N)~|~\prod_{m < n} |z_n - z_m|\prod_n w_K^{N-1}(z_n) \geq \exp\left\{-\big(I[\omega_K]+\varepsilon\big)\frac{N(N-1)}2\right\}\right\}.
\end{equation}

\begin{lemma}
\label{lem:lowenergy}
For each $N\in\N$ and $\varepsilon>0$, it holds that
\[
\int_{\C^N\setminus U_{N,\varepsilon}}\Omega_{N,s,\beta}\mathrm{d}A^{\otimes N} \leq \exp\left\{-\beta\big(\varepsilon + o(1)\big)\frac{N(N-1)}{2}\right\}.
\]
\end{lemma}
\begin{proof}
For $(z_1,\ldots,z_N)\in\C^N\setminus U_{N,\varepsilon}$ it holds that
\begin{eqnarray}
\Omega_{N,s,\beta}(z_1,\ldots,z_N) &\leq& \frac{1}{Z_{N,s\beta}}\exp\left\{-\beta\big(I[\omega_K]+\varepsilon\big)\frac{N(N-1)}2\right\}\prod_{n=1}^Nw_K^{\beta(s-N+1)}(z_n) \nonumber \\
&\leq& \exp\left\{-\beta\big(\varepsilon + o(1)\big)\frac{N(N-1)}2\right\}\prod_{n=1}^Nw_K^{2+c_0}(z_n), \nonumber
\end{eqnarray}
where we used \eqref{partitionlimit} and \eqref{snbeta} for the second inequality. Now, the conclusion of the lemma follows from the fact that
\[
\int_{\C^N\setminus U_{N,\varepsilon}}\prod_{n=1}^Nw_K^{2+c_0}(z_n)\mathrm{d}A^{\otimes N} \leq \left(\int_\C w_K^{2+c_0}\mathrm{d}A\right)^N = \exp\big\{\mathcal{O}(N)\big\}.
\]
\end{proof}

For $f\in\mathrm{C}_b(\C^n)$, set
\begin{equation}
\label{fN}
f_N(z_1,\ldots,z_N) := \frac{(N-n)!}{N!}\sum_\sigma f\left(z_{\sigma(1)},\ldots,z_{\sigma(n)}\right),
\end{equation}
where the sum is taken over all distinct permutations $\sigma$ of size $n$ for $\{1,\ldots,N\}$.

\begin{lemma}
\label{lem:fN}
Let $f\in\mathrm{C}_b(\C^n)$ and $\boldsymbol z=\{z_1,\ldots,z_N\}$. Then there exists a finite constant $c(n)$ such that
\[
\left|f_N(z_1,\ldots,z_N) - \int f\mathrm{d}\omega_{\boldsymbol z}^{\otimes n}\right| \leq \max_{\C^n}|f|\frac{c(n)}{N}.
\]
\end{lemma}
\begin{proof} 
For $n=1$ it simply holds that $f_N(z_1,\ldots,z_N) = \int f\mathrm{d}\omega_{\boldsymbol z}$. When $n=2$, it is true that
\[
\left|f_N(z_1,\ldots,z_N) - \int f\mathrm{d}\omega_{\boldsymbol z}^{\otimes 2}\right| = \left|\frac1{N-1}\int f\mathrm{d}\omega_{\boldsymbol z}^{\otimes 2} - \frac{1}{N-1}\int f(u,u)\mathrm{d}\omega_{\boldsymbol z}(u)\right| \leq \frac{2\max_{\C^n}|f|}{N-1}.
\]
More generally, it holds that $f_N(z_1,\ldots,z_N)$ is equal to
\[
\frac{(N-n)!}{N!}\left(N^n\int f\mathrm{d}\omega_{\boldsymbol z}^{\otimes n} - N^{n-1}\sum\int f\mathrm{d}\omega_{\boldsymbol z}^{\otimes n-1} + \cdots + (-1)^{n-1}N\int f(u,\ldots,u)\mathrm{d}\omega_{\boldsymbol z}(u)\right),
\]
where the first sum is taken over all possible combinations of two coordinates being equal, the next sum is taken over over all possible combinations of three coordinates being equal, etc. Since the number of terms in each sum depends on $n$ but is independent of $N$, the conclusion of the lemma follows.
\end{proof}

Denote by $\mathrm{C}_c(\C^n)$ the collection of continuous functions on $\C^n$ with compact support.

\begin{lemma}
\label{lem:cws}
For any $f\in\mathrm{C}_c(\C^n)$, it holds that
\[
\lim_{\varepsilon\to0}\lim_{N\to\infty}\sup_{(z_1,\ldots,z_N)\in U_{N,\varepsilon}}\left|f_N(z_1,\ldots,z_N)-\int f\mathrm{d}\omega_K^{\otimes n}\right| = 0.
\]
\end{lemma}
\begin{proof}
For fixed $\varepsilon$, let $\left\{\left(\lambda^{N,\varepsilon}_1,\ldots,\lambda^{N,\varepsilon}_N\right)\right\}_{n\in\Lambda}$, $\Lambda\subseteq\N$, be a maximizing sequence for the first limit in question. Since $f$ has compact support, we can assume that $|\lambda_k^{N,\varepsilon}|<R$ for some $R=R(f)$ large enough. Set
\[
\boldsymbol \lambda^{N,\varepsilon} := \left\{\lambda^{N,\varepsilon}_1,\ldots,\lambda^{N,\varepsilon}_N\right\}.
\]
Let $\nu_\varepsilon$ be a weak$^*$ limit point of $\left\{\omega_{\boldsymbol\lambda^{N,\varepsilon}}\right\}$. Clearly, $\nu_\varepsilon$ is also supported in the disk of radius $R$. Then it follows from Lemma~\ref{lem:fN}, the choice of $\left\{\boldsymbol \lambda^{N,\varepsilon}\right\}$ that
\[
\limsup_{N\to\infty}\sup_{(z_1,\ldots,z_N)\in U_{N,\varepsilon}}\left|f_N(z_1,\ldots,z_N) - \int f\mathrm{d}\omega_K^{\otimes n}\right| = \left|\int f\mathrm{d}\nu_\varepsilon^{\otimes n} - \int f\mathrm{d}\omega_K^{\otimes n}\right|
\]
(observe that $\omega_{\boldsymbol\lambda^{N,\varepsilon}}^{\otimes n}\cws\nu_\varepsilon^{\otimes n}$ since $\omega_{\boldsymbol\lambda^{N,\varepsilon}}\cws\nu_\varepsilon$). Further, let $\nu$ be a weak$^*$ limit point of $\{\nu_\varepsilon\}$ such that
\[
\limsup_{\varepsilon\to0}\lim_{N\to\infty}\sup_{(z_1,\ldots,z_N)\in U_{N,\varepsilon}}\left|f_N(z_1,\ldots,z_N)-\int f\mathrm{d}\omega_K^{\otimes n}\right| = \left|\int f\mathrm{d}\nu^{\otimes n} - \int f\mathrm{d}\omega_K^{\otimes n}\right|.
\]
The proof of the lemma will be completed if we show that $\nu=\omega_K$. To this end, recall that $\left(\lambda^{N,\varepsilon}_1,\ldots,\lambda^{N,\varepsilon}_N\right)\in U_{N,\varepsilon}$ and therefore
\[
I_*[\omega_{\boldsymbol\lambda^{N,\varepsilon}}] + 2\int g_K\mathrm{d}\omega_{\boldsymbol\lambda^{N,\varepsilon}} \leq I[\omega_K] + \varepsilon.
\]
Then it follows from Proposition~\ref{prop:minenergy}, Lemma~\ref{lem:descent}, and the inequality above that
\[
I[\omega_K] \leq I_1[\nu_\varepsilon] \leq I[\omega_K] + \varepsilon.
\]
Applying principle of descent once more, we get that
\[
I[\omega_K] \leq I_1[\nu]\leq \liminf I_1[\nu_\varepsilon] \leq I[\omega_K].
\]
The desired conclusion now follows from Proposition~\ref{prop:minenergy}.
\end{proof}

\begin{proof}[Proof of Theorem~\ref{thm:cws}] Fix $f\in\mathrm{C}_c(\C^n)$. That is, we assume that $f$ has compact support. Define $f_N$ by \eqref{fN}. Then
\begin{eqnarray}
\int_{\C^N} f_N\Omega_{N,s,\beta}\mathrm{d}A^{\otimes N} &=& \frac{(N-n)!}{N!}\sum_\sigma\int_{\C^N}f\left(z_{\sigma(1)},\ldots,z_{\sigma(n)}\right)\Omega_{N,s,\beta}(z_1,\ldots,z_N)\mathrm{d}A^{\otimes N}(z_1,\ldots,z_N) \nonumber \\
&=& \int_{\C^n} f\Omega^{(n)}_{N,s,\beta}\mathrm{d}A^{\otimes n} \nonumber
\end{eqnarray}
as $\Omega_{N,s,\beta}$ is symmetric with respect to the permutations of $z_1,\ldots,z_N$. Moreover, since $\Omega_{N,s,\beta}$ is a probability density function by its very definition, it holds that
\[
\int_{\C^n} f\Omega^{(n)}_{N,s,\beta}\mathrm{d}A^{\otimes n} - \int_{\C^n}f\mathrm{d}\omega_K^{\otimes n} = \int_{\C^N} \left(f_N- \int_{\C^n}f\mathrm{d}\omega_K^{\otimes n}\right)\Omega_{N,s,\beta}\mathrm{d}A^{\otimes N}.
\]
Hence, by Lemma~\ref{lem:lowenergy} we only need to show that
\[
\lim_{N\to\infty}\int_{U_{N,\varepsilon}} \left(f_N- \int_{\C^n}f\mathrm{d}\omega_K^{\otimes n}\right)\Omega_{N,s,\beta}\mathrm{d}A^{\otimes N} = 0
\]
for any $\varepsilon>0$ as $\sup_{\C^N}|f_N|=\sup_{\C^n}|f|$. The desired conclusion now follows from Lemma~\ref{lem:cws} since
\[
\int_{U_{N,\varepsilon}} \left|f_N- \int_{\C^n}f\mathrm{d}\omega_K^{\otimes n}\right|\Omega_{N,s,\beta}\mathrm{d}A^{\otimes N} \leq  \sup_{(z_1,\ldots,z_N)\in U_{N,\varepsilon}}\left|f_N(z_1,\ldots,z_N)-\int f\mathrm{d}\omega_K^{\otimes n}\right|,
\]
where we once more used the fact that $\Omega_{N,s,\beta}$ is positive and has unit integral over $\C^N$.

Now, let $D$ be a large enough ball in $\C^n$ to contain $K^n$ and $f$ be a function supported in $D$ satisfying $0\leq f\leq1$ and such that $f\equiv1$ on $K^n$. Then
\[
\limsup_{N\to\infty}\int_{D^c}\Omega^{(n)}_{N,s,\beta}\mathrm{d}A^{\otimes n} \leq 1 - \lim_{N\to\infty}\int_Df\Omega^{(n)}_{N,s,\beta}\mathrm{d}A^{\otimes n} = 1 - \int f\mathrm{d}\omega_K^{\otimes n} = 0.
\]
Since any $f\in\mathrm{C}_b(\C^n)$ can be written as a sum $f_c+(f-f_c)$, where $f_c\in\mathrm{C}_c(\C^n)$ and $f\equiv f_c$ in $D$, the general claim follows.
\end{proof}

\bibliographystyle{plain}
\bibliography{../bibliography}

\end{document}